\documentclass[notitlepage, oneside]{article}
\usepackage{color}
\usepackage[pdftex]{hyperref}

\usepackage[initials]{amsrefs}
\usepackage[margin=1in]{geometry}  

\usepackage{amsmath}
\usepackage{amsthm}	
\usepackage{graphicx}	
\usepackage{amsfonts}
\usepackage{amssymb}
\usepackage{bbm}
\usepackage{mathrsfs}

\theoremstyle{definition}
  	\newtheorem{theorem}{Theorem}
  	\newtheorem{lemma}[theorem]{Lemma}
	\newtheorem{proposition}[theorem]{Proposition}
  	\newtheorem{corollary}[theorem]{Corollary}
  	\newtheorem{definition}[theorem]{Definition}
  	\newtheorem{remark}[theorem]{Remark}

\newcommand{\R}{\mathbb R}
\newcommand{\Z}{\mathbb Z}

\newcommand{\C}{\mathbb C}

\newcommand{\E}{\mathbb{E}}
\newcommand{\I}{\infty}

\newcommand{\PB}{\mathbb{P}}
\newcommand{\1}{\mathbbm{1}}

\newcommand{\ud}{\, \mathrm{d}}

\DeclareMathOperator{\real}{Re}
\DeclareMathOperator{\imag}{Im}

\DeclareMathOperator{\Mat}{Mat}

\begin{document}

\title{Asymptotic Enumeration of Difference Matrices over Cyclic Groups}
\author{Aaron M. Montgomery \\ Department of Mathematics and Computer Science \\ Baldwin Wallace University}
\maketitle

\begin{abstract}
We identify a relationship between a certain family of random walks on Euclidean lattices and difference matrices over cyclic groups. We then use the techniques of Fourier analysis to estimate the return probabilities of these random walks, which in turn yields the asymptotic number of difference matrices over cyclic groups as the number of columns increases.

\end{abstract}

\section{Introduction}

This paper will explore a connection between random walks and a certain class of combinatorial designs known as difference matrices. Difference matrices have long been a part of the combinatorial design literature, and they are related to many other types of designs such as orthogonal arrays, transversal designs,  pairwise-balanced designs, and more. For instance, a difference matrix over $\Z_2$ is also a partial Hadamard matrix (see, for instance, \cite{wdl_levin}). A comprehensive overview of the existing literature on difference matrices and their relationships with other types of designs can be found in \cite{crc}.

The main result of this work will be to provide an asymptotic count of the difference matrices over cyclic groups as the number of their columns increases. Enumerating combinatorial designs directly is often quite challenging due to the computational complexity involved, and difference matrices are no exception. We will instead relate these matrices to a certain family of random walks on Euclidean lattices, and we will use the tools of Fourier analysis to  estimate the return probability of these random walks. We will then exploit the connection between the two problems to obtain the desired result about difference matrices. One advantage of this approach is that it permits an asymptotic enumeration of these matrices without ever requiring the explicit construction of any such matrix, which can be computationally difficult when the parameters are large.

\begin{definition} \label{defn_dm}
  Let $(G, \odot)$ be a group of order $g$. A \emph{$(g, k; \lambda)$-difference matrix over $G$} is a $k \times g \lambda$ matrix $D = [d_{ij}]$ with entries from $G$, so that for each $1 \leq i < j \leq k$, the multiset $\{d_{i \ell} \odot d_{j \ell}^{-1} : 1 \leq \ell \leq g \lambda\}$ contains every element of $G$ exactly $\lambda$ times.
\end{definition}

In order to simplify the Fourier analysis, in this work we will consider only cases where $G = \Z_g$. Accordingly, we will prefer to use notation such as $d_{i \ell} - d_{j \ell}$ in place of $d_{i \ell} \odot d_{j \ell}^{-1}$. The goal of this paper is to prove the following theorem.

\begin{theorem} \label{main_theorem}
Let $g \geq 2$ be fixed, and suppose $(k, \lambda)$ is a sequence of ordered pairs such that $k \geq 3$, $\lambda \to \I$, and if $g$  is even then each $\lambda$ is also even. Suppose also that there exists some $\varepsilon_0 > 0$ so that $k < (\frac 1 6 - \varepsilon_0) \frac{\log(\lambda g)}{\log(g)}$ for all pairs in the sequence. Then along the sequence $(k, \lambda)$ the number of $(g, k; \lambda)$-difference matrices over $\Z_g$ is 
  \[ \frac{ g^{k \lambda g + \frac{(3k-4)(k-1)}{4}}}{\sqrt{(2 \pi \lambda)^{\binom{k}{2}(g-1)}}}[1 + o(1)].\]
\end{theorem}

We remark here that the prohibition on the existence of a $(g, k; \lambda)$-difference matrix over $\Z_g$ with $k > 2$, $g$ even, and $\lambda$ odd is due to Drake \cite{drake}, although the Fourier analysis involved in the proof of Theorem \ref{main_theorem} will suggest this as well. We also note that Jungnickel showed that the existence of a $(g, k; \lambda)$-difference matrix over $\Z_g$ requires that $k \leq \lambda g$ \cite{Jungnickel1979}. Consequently, an ideal version of Theorem \ref{main_theorem} would permit $k$ to grow as a near-linear function of $\lambda$, whereas this version permits $k$ only to grow as a logarithmic function of $\lambda$. We do not claim any technical or number-theoretic reason for this restriction; rather, it is made as a matter of practicality to complete the Fourier analysis, and it could almost surely be improved.

To prove Theorem \ref{main_theorem}, we will consider a randomly-generated $k \times g \lambda$ matrix with entries chosen uniformly and independently from $\Z_g$; we will then compute the probability that such a matrix satisfies Definition \ref{defn_dm}. The columns of the matrix will correspond to the steps of a random walk on a certain high-dimensional Euclidean lattice, and the existence of a $(g, k; \lambda)$-difference matrix over $\Z_g$ will correspond to a return path of the random walk to the origin. We will prove a local central limit theorem to provide estimates on the return probability of the random walk, which will in turn yield estimates on the numbers of these matrices. 

We pause to remark that this type of analysis is certainly not new to the study of combinatorial designs. De Launey and Levin used this tactic to study partial Hadamard matrices \cite{wdl_levin}, and their enumeration results can be recognized as the particular case of Theorem \ref{main_theorem} where $g = 2$, although their work shows the formula to be valid so long as the pairs $(k, \lambda)$ satisfy the more generous condition $k \leq (2 \lambda)^{1/(12 + \varepsilon_0)}$. Additionally, the author of this work has used this strategy to count balanced incomplete block design incidence matrices \cite{montgomery_thesis}*{Thm 2.3}. Kuperberg, Lovett, and Peled  used a slightly different random walk approach to enumerate simple orthogonal arrays and simple $t$-designs \cite{klp}. Many other works exist which perform asymptotic enumerations of combinatorial structures by finding suitable estimates on complex integrals; see \cites{gmw, isaev, canfield_mckay, barvinok_hartigan, mckay_wormald} for a few of the numerous examples. 

We now define the random walk and identify its correspondence to difference matrices. With $(\Z_g)^k$ regarded as a column vector with entries in $\Z_g$, we will define a map $Z \colon (\Z_g)^k \to \R^{\binom k 2 \cdot (g-1)}$. The Euclidean space is to be regarded as indexed by two coordinates, where the first coordinate is an unordered pair of rows (i.e. $\{i, j\}$ with $1 \leq i < j \leq k$) and the second coordinate is an element of $\Z_g$ besides $0$. The ordering of the indices will be assumed to be lexicographic. Let $\vec x = (x_1, \dots, x_k)^T$ be an element of $(\Z_g)^k$; then the map $Z$ is defined by
\begin{equation} \label{defn_map}
  [Z(\vec x)]_{\{i,j\},a} =
    \begin{cases}
      1 -1/g, & \textrm{ if } x_i - x_j = a \\
      -1/g, & \textrm{ if } x_i - x_j \neq a.
    \end{cases}
\end{equation}
As an example, if $g = 3$ and $k = 4$, then the codomain $\R^{12}$ of the map $Z$ is understood to be indexed in the order $(\{1,2\},1), (\{1,2\},2), (\{1,3\},1), \dots, (\{3,4\},2);$ the vector $(2, 1, 0, 2)^T$ would map to $\frac{1}{3} \cdot (2, -1, -1, 2, -1, -1, 2, -1, -1, 2, 2, -1)$ under $Z$.
We remark that in the index scheme, for each pair of rows $\{i,j\}$, the coordinate $(\{i,j\},0)$ is intentionally omitted. The purpose of the function $Z$ is that if $D = [\vec x_1, \dots, \vec x_t]$ is a $k \times t$ matrix with entries in $\Z_g$, then a given coordinate $\{i, j\}, a$ of the expression $Z(\vec x_1) + \dots + Z(\vec x_t)$ will be $\vec 0$ if and only if $t$ is a multiple of $g$ and the multiset $\{x_{i \ell} - x_{j \ell} : 1 \leq \ell \leq t\}$ contains the element $a$ exactly $t/g$ times. Consequently, we have $Z(\vec x_1) + \dots + Z(\vec x_t) = \vec 0$ if and only if every element of $\{1, \dots, g-1\} \subset \Z_g$ appears exactly $t/g$ times in the multiset $\{x_{i \ell} - x_{j \ell} : 1 \leq \ell \leq t\}$ for every pair $\{i,j\}$; the latter condition also implies that $0$ also appears exactly $t/g$ times. 

\begin{definition}
  Let $\{X_t\}$ be the random walk defined with increments drawn uniformly and independently from $\{Z(\vec x): \vec x \in (\Z_g)^k\}$.
\end{definition}

If $D = [\vec x_1, \dots, \vec x_t]$ is a $k \times t$ matrix with elements taken uniformly and independently from $\Z_g$, then the preceding remarks imply that $D$ is a difference matrix if and only if $Z(\vec x_1) + \dots + Z(\vec x_t) = \vec 0$; that is, the corresponding random walk has returned to the origin. Since $Z$ is invariant under the action of adding any element to each entry of $\vec x$, there is a $g^t$ to $1$ correspondence between $\Mat_{k \times t}(\Z_g)$ and paths of the random walk $\{X_t\}$. There is also a $g^t$ to $1$ correspondence between $(g, k ; \lambda)$-difference matrices over $\Z_g$ and return paths of $\{X_t\}$ to $\vec 0$, which implies the following relationship:
\begin{equation*}
  \frac{\# \, (g, k; \lambda)\textrm{-difference matrices over $\Z_g$}}{|\Mat_{k \times \lambda g}(\Z_g)|} = \frac{\textrm{\# return paths of $X_t$ to $\vec 0$}}{\textrm{\# all paths of $X_t$}}
\end{equation*}
The right side of that equation is merely $\PB(X_t = \vec 0)$, and the denominator on the left is $g^{k \lambda g}$. Therefore, we have
\begin{equation} \label{walk_correspondence}
  \# \, (g, k; \lambda)\textrm{-difference matrices over $\Z_g$} = g^{k \lambda g} \PB(X_t = \vec 0).
\end{equation}
Hence, in order to estimate the number of $(g, k; \lambda)$-differences matrices over $\Z_g$, we need only to estimate the return probability of $\{X_t\}$ to $\vec 0$.

We will estimate $\PB(X_t = \vec 0)$ by using the standard tactics of Fourier analysis. In all that follows, we let $d = \binom k 2 \cdot (g-1)$. We define the characteristic function $\Phi: \R^d \to \C$ by
  \[ \Phi(\vec \theta) = \E \left[e^{i \vec \theta \cdot X_1}\right] = \sum_{\vec x \in (\Z_g)^k} g^{-k} e^{i \vec \theta \cdot Z(\vec x)}. \]
One can verify that $|\Phi(\vec \theta)|$ is $2\pi$-periodic, and that if $t$ is a multiple of $g$, then $\Phi(\vec \theta)^t$ is also $2\pi$-periodic.
 Consequently, if $t$ is a multiple of $g$, then the walk $X_t$ is supported on the integer lattice $\Z^d \subset \R^d$, which permits use of the Fourier inversion formula (see, for instance, \cite{spitzer}*{P3, p.\ 57}):
  \begin{equation} \label{fourier_inversion}
    \PB(X_t = \vec 0) = (2 \pi)^{-d} \int_{[-\pi, \pi)^d} \Phi(\vec \theta)^t \ud \vec \theta .
  \end{equation}
To estimate this integral, we will partition $[-\pi, \pi)^d$ based on the value of $|\Phi(\vec \theta)|$ by dealing separately with regions where $|\Phi(\vec \theta)|$ is close to $1$ and those where it is not. Intuitively, as $t \to \I$, the contributions to the integral from the former regions should become dominant, while those from the latter regions should become negligible. 

It is noteworthy that, in principle, combining \eqref{walk_correspondence} with \eqref{fourier_inversion} provides an exact count of the number of $(g, k; \lambda)$-difference matrices over $\Z_g$. However, in practice this is difficult to exploit because the 
complicated nature of $\Phi(\vec \theta)^t$ makes the integral intractibly difficult. Consequently, it will be preferable to estimate $\Phi(\vec \theta)^t$ instead of calculating it directly.
We also note that a common approach to the general problem of estimating return probabilities of a random walk is to transform the walk to a strongly aperiodic random walk on an integer lattice. However, this tactic is difficult here because of the complicated structure of the increment set $\{Z(\vec x): \vec x \in (\Z_g)^k\}$. These challenges motivate the Fourier-analytic approach used in this work.

The outline of the rest of this paper is as follows: in Section \ref{S: anatomy}, we discuss how to decompose the integral in \eqref{fourier_inversion} into manageable pieces. Sections \ref{S: remainder} and \ref{S: primary} are devoted to finding estimates on the integral where $|\Phi(\vec \theta)|$ is and is not close to $1$, respectively. Finally, in Section \ref{S: proof} we combine all the pieces to prove Theorem \ref{main_theorem}.
 
\section{Anatomy of the Integral} \label{S: anatomy}
Let $\Lambda = \{\vec \theta \in \R^d : |\Phi(\vec \theta)| = 1\}$. Our goal in this section is to characterize the set $\Lambda$, and then to use this characterization to partition the integral in \eqref{fourier_inversion} in a suitable way. We observe that
\[ \vec \theta \in \Lambda \iff \textrm{ for all } \vec x, \vec y \in (\Z_g)^k,\,  e^{i \vec \theta \cdot Z(\vec x)} = e^{i \vec \theta \cdot Z(\vec y)} \]
or equivalently, 
\begin{equation} \label{2pi_equivalent}
  \vec \theta \in \Lambda \iff \textrm{ for all } \vec x, \vec y \in (\Z_g)^k,\,  \vec \theta \cdot Z(\vec x) \equiv \vec \theta \cdot Z(\vec y) \pmod{2 \pi}.
\end{equation}
This expression shows that $\Lambda$ is closed under addition and under negation, so $\Lambda$ is a subgroup of the additive group $\R^d$. In the sequel, any reference to vectors in $\R^d$ being equivalent modulo $2 \pi$ should be understood to mean that their respective components are equivalent to one another modulo $2 \pi$. We remark that $\Phi^t$ is $2 \pi$-periodic in this sense if $t$ is a multiple of $g$.

In order to characterize $\Lambda$, we first observe a useful fact.
\begin{lemma} \label{taurus_lemma_general}
  Let $a, b \in \Z_g$ and let $1 \leq i < j \leq k$. Suppose that $\vec \theta \in \R^d$ and $\epsilon > 0$ have the property that for any pair of vectors $\vec x, \vec y \in (\Z_g)^k$, there are $z \in \Z$ and $\epsilon_1$ with $|\epsilon_1| < \epsilon$ such that $\vec \theta \cdot Z(\vec x) - \vec \theta \cdot Z(\vec y) = 2 \pi z + \epsilon_1$. Then there exists $\epsilon_2$ with $|\epsilon_2| < 2 \epsilon$ such that
    \[ \theta_{\{i,j\},a} + \theta_{\{i,j\},b} \equiv \theta_{\{i,j\},a+b} + \epsilon_2 \pmod{2 \pi}\]
where, if necessary, the undefined value $\theta_{\{i,j\},0}$ is understood to be 0.
\end{lemma}

The interpretation of this lemma is that if $\vec \theta \cdot Z(\vec x)$ is nearly independent (modulo $2 \pi$) of $\vec x$, then the coordinates of $\vec \theta$ nearly satisfy a certain group homorphism property. This lemma is quite technical, and we delay its proof to examine several useful corollaries. The first corollary is also quite technical and will be used in section \ref{S: remainder}.

\begin{corollary} \label{mod_2pi_g_general}
  Suppose $\vec \theta \in \R^d$ and $\epsilon > 0$ have the property that for any pair of vectors $\vec x, \vec y \in (\Z_g)^k$, there are $z \in \Z$ and $\epsilon_1$ with $|\epsilon_1| < \epsilon$ such that $\vec \theta \cdot Z(\vec x) - \vec \theta \cdot Z(\vec y) = 2 \pi z + \epsilon_1$. Then for any $\{i,j\},a$, there are $z \in \Z$ and $\epsilon_3$ with $|\epsilon_3| < 2 \epsilon$ such that 
  \[ \theta_{\{i,j\},a} - \epsilon_3 = \frac{2 \pi}{g} z.\]
\end{corollary}

\begin{proof}
  If $\vec \theta$ and $\epsilon$ satisfy the hypotheses of Lemma \ref{taurus_lemma_general}, then since $ga = 0$ it follows inductively that $g \theta_{\{i,j\},a} = 2 \pi z + \epsilon_{g}$, where $z \in \Z$ and $|\epsilon_g| < 2 (g - 1) \epsilon$ by the triangle inequality. If we set $\epsilon_3 = \epsilon_{g}/g$, then $\theta_{\{i,j\},a} = \frac{2 \pi}{g} z + \epsilon_3$, where $|\epsilon_3| < 2\frac{g-1}{g} \epsilon < 2\epsilon$.
\end{proof}

The second corollary can be obtained by letting $\epsilon \to 0$ in Lemma \ref{taurus_lemma_general}.

\begin{corollary}\label{taurus_lemma}
  Let $a, b \in \Z_g$ and let $1 \leq i < j \leq k$. If $\vec \theta \in \Lambda$, then
  \[ \theta_{\{i, j\}, a} + \theta_{\{i, j\}, b} \equiv \theta_{\{i, j\}, a + b} \pmod{2 \pi} \]
  where, if necessary, the undefined value $\theta_{\{i,j\},0}$ is understood to be 0.
\end{corollary}

The third corollary is an immediate consequence of Corollary \ref{taurus_lemma}.

\begin{corollary} \label{mod_2pi_g}
  If $\vec \theta \in \Lambda$, then for every $\{i, j\},a$, we have $\theta_{\{i,j\},a} \equiv 0 \pmod{\frac{2\pi}{\nu}}$, where $\nu$ is the order of $a$. In particular, for all nonzero elements $a$, it holds that $\theta_{\{i,j\},a} \equiv 0 \pmod{\frac{2\pi}{g}}$.
\end{corollary}

\begin{proof}[Proof of Lemma \ref{taurus_lemma_general}]
  Fix elements $a, b \in \Z_g$ and integers $i, j$ with $1 \leq i < j \leq k$. If either $a$ or $b$ is $0$, then the result is trivial; we assume first that $a, b,$ and $a + b$ are all nonzero. We define four vectors in $(\Z_g)^k$:
    \begin{align*}
      \vec x_1 = (0, \dots, 0, a, 0, \dots, 0&, 0, 0, \dots, 0)^T \\
      \vec x_2 = (0, \dots, 0, 0, 0, \dots, 0&, -b, 0, \dots, 0)^T \\
      \vec x_3 = (0, \dots, 0, a, 0, \dots, 0&, -b, 0, \dots, 0)^T \\
      \vec x_4 = (0, \dots, 0, 0, 0, \dots, 0&, 0, 0, \dots, 0)^T
    \end{align*}
The $a$ in $\vec x_1$ and $\vec x_3$ occurs in the $i^{\textrm{th}}$ position, and the $-b$ in $\vec x_2$ and $\vec x_3$ occurs in the $j^{\textrm{th}}$ position. Let $\vec \theta \in \R^d$. We will use $\vec 1$ to represent the vector in $\R^d$ of all ones. The following calculations are straightforward from Definition \ref{defn_map}:
  \begin{align*}
    \vec \theta \cdot Z(\vec x_1) &= \sum_{m < i} \theta_{\{m,i\},-a} + \sum_{m > i} \theta_{\{i,m\},a} - \frac{1}{g} \vec \theta \cdot \vec 1 \\
    \vec \theta \cdot Z(\vec x_2) &= \sum_{m < j} \theta_{\{m,j\},b} + \sum_{m > j} \theta_{\{j,m\},-b} - \frac{1}{g} \vec \theta \cdot \vec 1 \\
    \vec \theta \cdot Z(\vec x_3) &= \sum_{m < i} \theta_{\{m,i\},-a} + \sum_{\substack{m > i \\ m \neq j}} \theta_{\{i,m\},a} + \sum_{\substack{m < j \\ m \neq i}} \theta_{\{m,j\},b} + 
    \sum_{m > j} \theta_{\{j,m\},-b} + \theta_{\{i,j\},a+b}- \frac{1}{g} \vec \theta \cdot \vec 1 \\
    \vec \theta \cdot Z(\vec x_4) &= -\frac{1}{g} \vec \theta \cdot \vec 1
  \end{align*}
By assumption, there are integers $z, z'$ and error terms $\epsilon_1, \epsilon_1'$ in $(-\epsilon, \epsilon)$ such that $\vec \theta \cdot Z(\vec x_1) - \vec \theta \cdot Z(\vec x_3) = 2 \pi z + \epsilon_1$ and $\vec \theta \cdot Z(\vec x_2) - \vec \theta \cdot Z(\vec x_4) = 2 \pi z' + \epsilon_1'$. If $z'' = z + z'$ and $\epsilon_2 = \epsilon_1 + \epsilon'_1$, then 
\[ \vec \theta \cdot Z(\vec x_1) + \vec \theta \cdot Z(\vec x_2) - \vec \theta \cdot Z(\vec x_3) - \vec \theta \cdot Z(\vec x_4) = 2 \pi z'' + \epsilon_2\]
and $|\epsilon_2| < 2 \epsilon$ by the triangle inequality. Using the calculations of the dot product terms and cancelling all the relevant terms gives the desired result.

Finally, in the case where $a, b \neq 0$ but $a + b = 0$, the above proof still holds if the $\theta_{\{i,j\},a+b}$ term is omitted from the calculation of $\vec \theta \cdot Z(\vec x_3)$.
\end{proof}

We now collect another lemma regarding the structure of $\Lambda$.

\begin{lemma} \label{transcending_rows}
  Let $a$ be a nonzero element of $\Z_g$, and let $i$ be some fixed number between $1$ and $k$ (inclusively). If $\vec \theta \in \Lambda$, then
  \[   \sum_{1 \leq m < i} \theta_{\{m,i\},-a} + \sum_{i < m \leq k} \theta_{\{i,m\},a} \equiv 0 \pmod{2 \pi}\]
  where the appropriate empty sum is $0$ in the case that $i = 1$ or $i = k$.
\end{lemma}

\begin{proof}
  We define two vectors in $(\Z_g)^k$:
    \begin{align*}
      \vec x_1 &= (0, \dots, 0, a, 0, \dots, 0)^T \\
      \vec x_2 &= (0, \dots, 0, 0, 0, \dots, 0)^T
    \end{align*}
  The $a$ that appears in $\vec x_1$ does so in the $i^{\textrm{th}}$ position. If $\vec \theta \in \R^d$, then again with $\vec 1$ representing the vector in $\R^d$ of all ones, we have
    \begin{align*}
      \vec \theta \cdot Z(\vec x_1) &= \sum_{1 \leq m < i} \theta_{\{m,i\},-a} + \sum_{i < m \leq k} \theta_{\{i,m\},a} - \frac{1}{g} \vec \theta \cdot \vec 1 \\
      \vec \theta \cdot Z(\vec x_2) &= -\frac{1}{g} \vec \theta \cdot \vec 1
    \end{align*}
so if $\vec \theta \in \Lambda$, then by \eqref{2pi_equivalent} these two expressions are equivalent modulo $2\pi$, as desired.
\end{proof}

\begin{remark}
  We point out that the assumption that the underlying group $G$ is cyclic is not meaningfully used in the proofs of Lemma \ref{taurus_lemma_general} or any of its corollaries, nor is it used in the proof of Lemma \ref{transcending_rows}. In fact, even if $G$ is non-abelian (and the map $Z$ is redefined appropriately), these proofs require only trivial modifications such as exchanging $0 \in \Z_g$ for the identity element of $G$.
\end{remark}

In order to characterize $\Lambda$, we define a collection of ``building block'' vectors.

\begin{definition} \label{building_block}
  Fix a pair $\{i,j\}$ with $1 \leq i < j < k$ (note the strict inequality $j < k$). Let $\1_{\{i,j\},a}$ denote the vector in $\R^d$ with a $1$ in the $\{i,j\},a$ component and 0 elsewhere. We define the vector $\vec \alpha^{\{i,j\}}$ as follows:
    \[ \vec \alpha^{\{i,j\}} = \sum_{n=1}^{g-1} \frac{2\pi n}{g} \1_{\{i,j\},n}+ \sum_{n=1}^{g-1} \frac{2 \pi(g-n)}{g} \1_{\{i,k\},n} + \sum_{n=1}^{g-1} \frac{2 \pi n}{g} \1_{\{j,k\},n}\]
Here, $n$ is regarded to be an integer in $[0,g-1]$, except where it appears in the subscript of $\1$ as the corresponding element of $\Z_g$.
\end{definition}

\begin{proposition} \label{building_blocks_are_in}
  Any $\vec \alpha^{\{i,j\}}$ vector as defined in Definition \ref{building_block} is an element of $\Lambda$.
\end{proposition}

\begin{proof}
  Our goal is to show that the expression $Z(\vec x) \cdot \vec \alpha^{\{i,j\}}$ does not depend on $\vec x$. Using \eqref{defn_map} and Definition \ref{building_block}, we see that
  \begin{align*} Z(\vec x) \cdot \vec \alpha^{\{i,j\}} &\equiv \frac{2 \pi (x_i - x_j)}{g} + \frac{2 \pi (g - (x_i - x_k))}{g} + \frac{2 \pi(x_j - x_k)}{g} - 3 \sum_{n=1}^{g-1} \frac{2 \pi n}{g^2} \\
  & \equiv - 3 \sum_{n=1}^{g-1} \frac{2 \pi n}{g^2}  \pmod{2 \pi}.
  \end{align*}
which gives the desired result. As an addendum, we remark that
  \[
    -\sum_{n=1}^{g-1} \frac{2 \pi n}{g^2} = -\frac{2 \pi}{g^2} \frac{g(g-1)}{2} = -\frac{\pi(g-1)}{g}
  \]
whence
  \begin{equation} \label{building_block_dot_prod}
    Z(\vec x) \cdot \vec \alpha^{\{i,j\}} \equiv -\frac{ 3 \pi(g-1)}{g} \pmod{2 \pi} 
  \end{equation}
which we preserve for later use.
\end{proof} 



We are now prepared to fully characterize the set $\Lambda$. Since $|\Phi|$ is $2 \pi$-periodic, it will suffice to charactierize $\Lambda$ on the region $[-\pi, \pi)^d$; to that end, we let $\Lambda_0$ denote $\Lambda \cap [-\pi, \pi)^d$.

\begin{lemma} \label{Lambda_characterization}
  Let $\vec \theta \in \Lambda_0$ and let $\vec \alpha^{\{i,j\}}$ be as defined in Definition \ref{building_block}. There are constants $c_{\{i,j\}}$ with $1 \leq i < j < k$ such that $c_{\{i,j\}} \in \{0, 1, \dots, g-1\}$ and
  \[\vec \theta \equiv \sum_{1 \leq i < j < k} c_{\{i,j\}} \vec \alpha^{\{i,j\}} \pmod{2 \pi}. \]
Moreover, this representation of $\vec \theta$ is unique. 
\end{lemma}

\begin{proof}
Let $\vec \theta \in \Lambda_0$. Set each coefficient $c_{\{i,j\}}$ to be 
  \[ c_{\{i,j\}} = \begin{cases}
    \frac{g}{2 \pi} \theta_{\{i,j\},1}, & \textrm{ if } \theta_{\{i,j\},1} \geq 0 \\
    \frac{g}{2 \pi} \theta_{\{i,j\},1} + g, & \textrm{ if } \theta_{\{i,j\},1} < 0
  \end{cases} \]
and let 
    \[ \vec x =\left[ \sum_{1 \leq i < j < k} c_{\{i,j\}} \vec \alpha^{\{i,j\}} \right] - \vec \theta. \]
  We will show that $\vec x \equiv 0 \pmod{2 \pi}$. We first note that setting the $c_{\{i,j\}}$ terms in this way implies that each is in the set $\{0, 1, \dots, g-1\}$ by Corollary \ref{mod_2pi_g}. Since $\Lambda$ is a subgroup of $\R^d$, and since the bracketed term is in $\Lambda$ by Proposition \ref{building_blocks_are_in}, so also is $\vec x$.
 Among the collection $\{\vec \alpha^{\{i,j\}} : i < j < k\}$, only $\vec \alpha^{\{1,2\}}$ has a nonzero $\{1,2\},1$ component analogous comments apply for any other fixed pair $\{i,j\}$ with $i < j < k$. Thus, for $i < j < k$, we have that $x_{\{i,j\},1}$ is either $0$ (if $\theta_{\{i,j\},1} \geq 0$) or $2 \pi$ (if $\theta_{\{i,j\},1} < 0$). 
  By Corollary \ref{taurus_lemma}, this inductively implies that $x_{\{i,j\},a} \equiv 0 \pmod{2\pi}$ for all $a$ and all pairs $\{i,j\}$ with $i < j < k$. Hence, modulo $2\pi$, the only possible nonzero coordinates of $\vec x$ are those of the form $x_{\{i,k\},a}$ for some $i$ and $a$. 
        
  Next, fix $i$ between $1$ and $k-1$ (inclusively). By Lemma \ref{transcending_rows}, we see that
  \[ x_{\{i,k\},1} \equiv -\sum_{1 \leq m < i} x_{\{m,i\},-1} - \sum_{i < m < k} x_{\{i,m\},1} \pmod{2\pi}. \]
  We have already established that each summand on the right-hand side is $0 \pmod{2\pi}$, since it is not of the form $x_{\{i,k\},a}$. Thus, $x_{\{i,k\},1} \equiv 0$ as well. By Corollary \ref{taurus_lemma} again, this establishes that $x_{\{i,k\}, a} \equiv 0$ for all $a$. The fact that this occurs for all $i$ between $1$ and $k-1$ completes the argument that $\vec x \equiv 0$. Finally, the uniqueness of this expression is immediate from the fact that each distinct vector $\vec \alpha^{\{i,j\}}$ is the unique contributor to the $\theta_{\{i,j\},1}$ component.
\end{proof}

This characterization of $\Lambda_0$ motivates how we will break up the integral in \eqref{fourier_inversion}; the primary contribution to the integral will be the regions in $[-\pi, \pi)^d$ that are close to $\Lambda$. For $\vec \theta \in [-\pi, \pi)^d$, we define the box
  \begin{equation*} 
    B_\delta(\vec \theta) = \{\vec \mu \in [-\pi,\pi)^d : \vec \mu \equiv \vec \theta + \vec \zeta \textrm{ with } |\zeta_{\{i,j\},a}| < \delta \textrm{ for all } (\{i,j\},a) \}
  \end{equation*}
where the equivalence is taken modulo $2 \pi$. The parameter $\delta$ is assumed to be small and positive. We remark that if $\delta < \pi/g$, and $\vec \theta_1, \vec \theta_2$ are distinct elements of $\Lambda_0$, then $B_{\delta}(\vec \theta_1)$ and $B_{\delta}(\vec \theta_2)$ are disjoint by Corollary \ref{mod_2pi_g}.
Since we expect the bulk of the integral to be contributed from regions of the form $\{B_{\delta}(\vec \theta): \vec \theta \in \Lambda_0\}$, we define the ``remainder set''
  \begin{equation} \label{remainder_defn}
    R_{\delta} = [-\pi, \pi)^d \setminus \bigcup_{\vec \theta \in \Lambda_0} B_{\delta}(\vec \theta).
  \end{equation}

\begin{proposition} \label{useful_integral}
  Suppose $\delta < \pi/g$ and that $t$ is a multiple of $g$. If $g$ is odd, or if both $g$ and $t/g$ are even, then 
    \[ \PB(X_t = \vec 0) = \frac{g^{\binom{k-1}{2}}}{(2 \pi)^d} \int_{B_{\delta}(\vec 0)} \Phi(\vec \theta)^t \ud \vec \theta + \frac{1}{(2 \pi)^{d}} \int_{R_{\delta}} \Phi(\vec \theta)^t \ud \vec \theta \]
  whereas if $g$ is even and $t/g$ is odd, then
    \[ \PB(X_t = \vec 0) = \frac{1}{(2 \pi)^{d}} \int_{R_{\delta}} \Phi(\vec \theta)^t \ud \vec \theta. \]
\end{proposition}

\begin{proof}
  We first note that $\delta < \pi/g$ implies that distinct vectors $\vec \eta \in \Lambda_0$ have disjoint boxes $B_{\delta}(\vec \eta)$. From this fact and \eqref{fourier_inversion}, we have
  \begin{equation} \label{integral_breakup_1}
     \PB(X_t = \vec 0) = (2 \pi)^{-d} \sum_{\vec \eta \in \Lambda_0} \int_{B_{\delta}(\vec \eta)} \Phi(\vec \theta)^t \ud \vec \theta + (2 \pi)^{-d} \int_{R_\delta} \Phi(\vec \theta)^t \ud \vec \theta .
  \end{equation}
Next, we consider a nonzero $\vec \eta \in \Lambda_0$. By \eqref{2pi_equivalent}, we note that $\vec \eta \cdot Z(\vec x)$ is deterministic in that it does not depend on the random vector $\vec x \in (\Z_g)^k$. Hence, $e^{i \vec \eta \cdot X_1}$ is also deterministic, and we have $\Phi(\vec \eta) = \E[e^{i \vec \eta \cdot X_1}] = e^{i \vec \eta \cdot X_1}$. For any other $\vec \zeta \in \R^d$, it follows that 
  \begin{align*}
    \Phi(\vec \eta + \vec \zeta) &= \E[e^{i (\vec \eta + \vec \zeta) \cdot X_1}] \\
    &= e^{i \vec \eta \cdot X_1} \E[e^{i \vec \zeta \cdot X_1}] \\
    & = \Phi(\vec \eta) \Phi(\vec \zeta).
  \end{align*}
Since $\Phi^t$ is $2\pi$-periodic when $t$ is a multiple of $g$, then for $\vec \eta \in \Lambda_0$ we have
  \[ \int_{B_{\delta}(\vec \eta)} \Phi(\vec \theta)^t \ud \vec \theta =  \Phi(\vec \eta)^t \int_{B_{\delta}(\vec 0)} \Phi(\vec \theta)^t \ud \vec \theta\]
so \eqref{integral_breakup_1} becomes
  \begin{equation} \label{integral_breakup_2}
    \PB(X_t = \vec 0) = \frac{\sum_{\vec \eta \in \Lambda_0} \Phi(\vec \eta)^t }{(2 \pi)^d} \int_{B_{\delta}(\vec 0)} \Phi(\vec \theta)^t \ud \vec \theta + \frac{1}{(2\pi)^d} \int_{R_{\delta}} \Phi(\vec \theta)^t \ud \vec \theta
  \end{equation}
and our only remaining task is therefore to evaluate the sum $\sum_{\vec \eta \in \Lambda_0} \Phi(\vec \eta)^t$.

Since $\vec \eta \in \Lambda_0$, by Lemma \ref{Lambda_characterization} there are constants $c_{\{i,j\}}$ and $\vec z \in (2 \pi \Z)^d$ so that $\vec \eta = \sum_{1 \leq i < j < k} c_{\{i,j\}} \vec \alpha^{\{i,j\}} + \vec z$. Combining this with \eqref{building_block_dot_prod} shows that $\vec \eta \cdot Z(\vec x) \equiv \sum_{1 \leq i < j < k} -c_{\{i,j\}} \frac{3\pi(g-1)}{g} + \vec z \cdot Z(\vec x) \pmod{2 \pi}$. If $t = \lambda g$, then 
  \begin{equation*}
    \Phi(\vec \eta)^t = e^{i t \vec \eta \cdot X_1}
     = \exp \left[i \left(-3\pi (g-1) \lambda \sum_{1 \leq i < j < k} c_{\{i,j\}} + \lambda g \vec z \cdot Z(\vec x) \right) \right].
  \end{equation*}
Because $Z(\vec x) \in \left( \frac{1}{g} \Z \right)^d$ and $\vec z \in (2 \pi \Z)^d$, it follows that $\vec z \cdot Z(\vec x)$ is an integer multiple of $\frac{2 \pi}{g}$. Therefore, $\lambda g \vec z \cdot Z(\vec x)$ is an integer multiple of $2\pi$, and thus,
  \begin{equation*} 
    \Phi(\vec \eta)^t 
     = \exp \left[i \left(- 3 \pi (g-1) \lambda \sum_{1 \leq i < j < k} c_{\{i,j\}} \right) \right].
  \end{equation*}    
In the case where $g$ is odd, or if $\lambda = t/g$ is even, then the term in parentheses is an integer multiple of $2 \pi $, whence $\Phi(\vec \eta)^t = 1$ for every $\vec \eta \in \Lambda_0$. The uniqueness of the representation in Lemma \ref{Lambda_characterization} shows that $|\Lambda_0| =  g^{\binom{k-1}{2}}$, which is equal to $\sum_{\vec \eta \in \Lambda_0} \Phi(\vec \eta)^t$ in \eqref{integral_breakup_2}, as desired. 

On the other hand, if $g$ is even and $\lambda$ is odd, then $-3 \pi(g-1)\lambda \sum c_{\{i,j\}}$ must be congruent modulo $2\pi$ to $0$ or $\pi$. In this case, half of the choices for the coefficients $c_{\{i,j\}}$ have odd parity and the other half have even parity. The same is true of the collection of all possible sums $\sum c_{\{i,j\}}$, so half the terms in $\sum_{\vec \eta \in \Lambda_0} \Phi(\vec \eta)^t$ will be $1$ and the other half will be $-1$, and the sum will be 0.
\end{proof}

\section{Bounds in the Remainder Region} \label{S: remainder}

This section is devoted to obtaining upper bounds on $|\Phi(\vec \theta)|$ for $\vec \theta \in R_{\delta}$ and leveraging them to obtain an upper bound on the corresponding integral in Proposition \ref{useful_integral}. In all that follows, we assume that $\delta < \pi/g$. We begin by defining a useful set:
\begin{equation*} 
  L = \{\vec \theta \in [-\pi, \pi)^d : \theta_{\{i,j\},a} \equiv 0 \pmod{2 \pi / g} \textrm{ for all } \{i,j\},a  \}
\end{equation*}
We recall from Corollary \ref{mod_2pi_g} that $\Lambda_0 \subset L$. Using the boxes defined in \eqref{remainder_defn},  we note that if $\vec \eta_1, \vec \eta_2$ are distinct elements of $L$, then $\delta < \pi/g$ implies that $B_{\delta}(\vec \eta_1) \cap B_{\delta}(\vec \eta_2) = \emptyset$. We set $R_{\delta}^A = \bigcup_{\vec \eta \in L \setminus \Lambda_0} B_{\delta}(\vec \eta)$ and $R_{\delta}^B = [-\pi, \pi)^d \setminus \bigcup_{\vec \eta \in L} B_{\delta} (\vec \eta)$, which yields
  \begin{equation} \label{rmd_decomposition}
    R_{\delta} = R_{\delta}^A \cup R_{\delta}^B.
  \end{equation} Intuitively, $R_{\delta}^A$ is the portion of the remainder region that is close to satisfying the modular condition of Corollary \ref{mod_2pi_g} (but is not near $\Lambda$), and $R_{\delta}^B$ is the region that is far from satisfying the modular condition.

\begin{lemma} \label{remainder_A_bd}
  Suppose $\delta < \frac{8}{5} g^{-k-3} k^{-2}$. Then if $\vec \theta \in R_{\delta}^A$, we have
  \[ |\Phi(\vec \theta)| \leq 1 - \frac{1}{10} g^{-k-2}.\] 
\end{lemma}

\begin{proof}
  Suppose $\vec \eta \in L \setminus \Lambda_0$ and that $\vec \theta \in B_{\delta}(\vec \eta)$; then $\vec \theta$ is equivalent (modulo $2 \pi$) to $\vec \eta + \vec \zeta$ where $|\zeta_{\{i,j\},a}|<\delta$ for all $\{i,j\},a$. Since $\eta \in L \setminus \Lambda_0$, by \eqref{2pi_equivalent} there are $\vec x, \vec y \in (\Z_g)^k$ such that $\vec \eta \cdot (Z(\vec x) - Z(\vec y)) \equiv \frac{2\pi n}{g} \pmod{2\pi}$ with $n \in \{1, \dots, g-1\}$. It follows that
    \begin{equation*}
     |\Phi(\vec \eta)| = \left| \sum_{\vec w \in (\Z_g)^k} g^{-k} e^{i \vec \eta \cdot Z(\vec w)} \right|  \leq g^{-k} \left[ \left|e^{i \vec \eta \cdot Z(\vec x)} + e^{i \vec \eta \cdot Z(\vec y)}\right| + \left| \sum_{ \vec w \neq \vec x, \vec y} e^{i \vec \eta \cdot Z(\vec w)} \right|  \right]
     \end{equation*}
and since $|e^{i a} + e^{i b}|^2 = 2 + 2\cos(a-b)$, we have
  \begin{equation*}
    |\Phi(\vec \eta)| \leq g^{-k}\left[\sqrt{2 + 2 \cos \left( \frac{2 \pi}{g} \right)} + g^k - 2 \right].
  \end{equation*}
Applying the bounds $\sqrt{x} \leq 1 + x/4$ and $\cos(x) \leq 1 - x^2/2 + x^4 /24$ to this expression yields
     \begin{equation} \label{rmd_A_bd_begrudgingly}
       |\Phi(\vec \eta)| \leq 1 - g^{-k} \left[ \left(\frac{\pi}{g}\right)^2 - \frac{\left( \pi/g \right)^4}{3} \right].
     \end{equation}
Together with the fact that $ \left( \frac{\pi}{g} \right)^4 \leq \left( \frac{\pi}{2} \right)^2 \left( \frac{\pi}{g} \right)^2$ when $g \geq 2$, we have $ |\Phi(\vec \eta)| \leq 1 - g^{-k} \left( \frac{\pi}{g} \right)^2 \left[ 1 - \frac{(\pi/2)^2}{3 }\right]$. Hence,
  \begin{equation} \label{bad_region_bd_1}
    |\Phi(\vec \eta)| \leq 1 - \frac{1}{10} g^{-k} \left( \frac{ \pi}{g} \right)^2 \leq 1 - \frac{9}{10} g^{-k -2}.
  \end{equation}
We also note that $\vec \theta = \vec \eta + \vec \zeta + \vec z$ with $\vec z \in (2 \pi \Z)^d$, and since $|\Phi|$ is $2\pi$-periodic, it follows that 
\begin{equation} \label{technical_correction}
  |\Phi(\vec \theta)| = |\Phi(\vec \eta + \zeta)|.
\end{equation}
We now state a pair of remainder bounds on Taylor polynomials for $e^z$ that we will use here and elsewhere: if $a \geq 0$ and $b$ is real, then
  \begin{align}
    \left| e^{-a} - \sum_{s=0}^j \frac{(-a)^s}{s!} \right| & \leq \min \left\{ \frac{2|a|^j}{j!}, \frac{|a|^{j+1}}{(j+1)!} \right\}, \label{taylor_bd_1}\\
    \left| e^{ib} - \sum_{s=0}^j \frac{(ib)^s}{s!} \right| & \leq \min \left \{ \frac{2|b|^j}{j!}, \frac{|b|^{j+1}}{(j+1)!} \right\}. \label{taylor_bd_2}
  \end{align}
For a reference, \eqref{taylor_bd_1} can be found as \cite{billingsley}*{equation 26.4}; \eqref{taylor_bd_2} is proved similarly. From the triangle inequality, we see that
  \begin{align*}
    |\Phi(\vec \vec \eta + \vec \zeta) - \Phi(\vec \eta)| &= g^{-k} \left| \sum_{\vec w \in (\Z_g)^k} e^{i (\vec \eta + \vec \zeta) \cdot Z(\vec w)} - e^{i \vec \eta \cdot Z(\vec w)} \right| \\
    & \leq g^{-k} \sum_{\vec w \in (\Z_g)^k} \left| e^{i \vec \eta \cdot Z(\vec w)} \right| \left| e^{i \vec \zeta \cdot Z(\vec w)} - 1 \right|.
  \end{align*}
Since $|e^{i \vec \eta \cdot Z(\vec w)}| = 1$, applying \eqref{taylor_bd_2} to $|e^{i \vec \zeta \cdot Z(\vec w)} - 1|$ shows that 
  \begin{equation} \label{bad_region_interim1}
    |\Phi(\vec \vec \eta + \vec \zeta) - \Phi(\vec \eta)| \leq g^{-k} \sum_{\vec w \in (\Z_g)^k} |\vec \zeta \cdot Z(\vec w)|.
  \end{equation}
When considering an individual term $|\vec \zeta \cdot Z(\vec w)|$, we observe that $|\zeta_{\{i,j\},a}| < \delta$, and each coordinate of $Z(\vec w)$ is at most $1 - 1/g$. These vectors have $(g-1) \binom{k}{2}$ coordinates, so $|\vec \zeta \cdot Z(\vec w)| \leq (g-1) \binom{k}{2}\delta (1 - 1/g) = \frac{1}{g}(g-1)^2 \frac{k(k-1)}{2} \delta$. There are $g^k$ terms in the summation in \eqref{bad_region_interim1}, and therefore we have 
  \begin{equation} \label{bad_region_interim2}
    |\Phi(\vec \vec \eta + \vec \zeta) - \Phi(\vec \eta)| \leq \frac{1}{g}(g-1)^2 \frac{k(k-1)}{2} \delta. 
  \end{equation}
  Our assumption that $\delta \leq \frac{8}{5} g^{-k-3}k^{-2}$ implies that $\frac{1}{g}(g-1)^2 \binom{k}{2} \delta \leq \frac{4}{5} g^{-k-2}$. Combining this with \eqref{technical_correction}, \eqref{bad_region_bd_1}, and \eqref{bad_region_interim2} shows that
  \( |\Phi(\vec \theta)| = |\Phi(\vec \eta + \vec \zeta)| \leq |\Phi(\vec \eta)| + |\Phi(\vec \eta + \vec \zeta) - \Phi(\vec \eta)| \leq 1 - \frac{9}{10}g^{-k-2} + \frac{4}{5}g^{-k-2},\)
as desired.
\end{proof}

Next, we bound $|\Phi(\vec \theta)|$ on the region $R^B_{\delta}$.

\begin{lemma} \label{remainder_B_bd}
  Suppose $\delta < \pi/g$. Then if $\vec \theta \in R_{\delta}^B$, we have
    \[ |\Phi(\vec \theta)| \leq 1 - \frac{11}{48} g^{-k} \left( \frac{\delta}{2} \right)^2.\]
\end{lemma}

\begin{proof}
  If $\vec \theta \in R_{\delta}^B$, then there is some $\{i,j\},a$ such that $\theta_{\{i,j\},a}$ is not expressable as $\frac{2\pi n }{g} + \zeta$ with $n \in \Z$ and $|\zeta| < \delta$. By Corollary \ref{mod_2pi_g_general}, this means that there is a pair of vectors $\vec x, \vec y \in (\Z_g)^k$ such that for any $z \in \Z$, the equation \( \vec \theta \cdot Z(\vec x) - \vec \theta \cdot Z(\vec y) = 2 \pi z + \epsilon_1 \) requires that $|\epsilon_1| > \delta/2$; in other words, $\vec \theta \cdot Z(\vec x) $ and $\vec \theta \cdot Z(\vec y) $ are not within $\delta/2$ of each other when taken modulo $2\pi$. If we choose $z$ and $\epsilon_1$ so that $|\epsilon_1| \leq \pi$, then it follows that
  \[\cos(\vec \theta \cdot Z(\vec x) - \vec \theta \cdot Z(\vec y)) = \cos(\epsilon_1) \leq \cos(\delta / 2). \]
By repeating the same arguments that led to \eqref{rmd_A_bd_begrudgingly} with this $\vec x$ and $\vec y$, it follows that
  \[|\Phi(\vec \theta)| \leq 1 - g^{-k} \left[ \frac{(\delta/2)^2}{4} - \frac{(\delta/2)^4}{48} \right] \]
and because $\delta < 2$, the desired result follows.
\end{proof}

Finally, we put together the bounds on $|\Phi(\vec \theta)|$ over $R^A_{\delta}$ and $R^B_{\delta}$ to obtain the bound on the $R_{\delta}$ component of Proposition \ref{useful_integral}.

\begin{proposition} \label{remainder_bound}
  If $\delta < \frac{8}{5} g^{-k-3}k^{-2}$, then 
   \[ \left|\frac{1}{(2 \pi)^{d}} \int_{R_{\delta}} \Phi(\vec \theta)^t \ud \vec \theta \right| \leq \exp \left(- \frac{11}{192} g^{-k} t \delta^2  \right).\]
\end{proposition}

\begin{proof}
  Suppose $\delta < \frac{8}{5} g^{-k-3} k^{-2}$. Let $\vec \theta \in R_{\delta}$; since $\delta < \pi/g$, by \eqref{rmd_decomposition} either $\vec \theta \in R_{\delta}^A$ or $\vec \theta \in R_{\delta}^B$, and the hypotheses of both Lemmas \ref{remainder_A_bd} and \ref{remainder_B_bd} apply. The assumptions on $\delta$ imply also that $\delta < 1/g$, so the bound on $|\Phi(\vec \theta)|$ in Lemma \ref{remainder_B_bd} is higher than that of Lemma \ref{remainder_A_bd}; hence, we can assert that $|\Phi(\vec \theta)| \leq 1 - \frac{11}{48} g^{-k} \left(\frac{\delta}{2} \right)^2$. Since $R_\delta \subset [-\pi, \pi)^d$ and $1 - x \leq e^{-x}$, we conclude that
  \begin{align*}
    \left| \frac{1}{(2\pi)^d} \int_{R_{\delta}} \Phi(\vec \theta)^t \ud \vec \theta \right| &\leq \frac{1}{(2 \pi)^d} \int_{R_{\delta}} |\Phi(\vec \theta)^t| \ud \vec \theta \\
    & \leq \left[  1 - \frac{11}{48} g^{-k} \left( \frac{\delta}{2} \right)^2 \right]^t \\
    & \leq \exp \left( - \frac{11}{192}t g^{-k} \delta^2 \right)
  \end{align*}
as desired.
\end{proof}

\section{Bounds in the Primary Region} \label{S: primary}

The goal of this section is to obtain sharp estimates on $|\Phi(\vec \theta)|$ in the region $B_{\delta}(\vec 0)$. Our first task is to calculate and estimate quantities of the form $\E[(\vec \theta \cdot Z(\vec x))^p]$ for $p = 1, 2, 3, 4$, where the expectation denotes that $\vec x$ is drawn randomly and uniformly from $(\Z_g)^k$. 

\begin{proposition} \label{first_moment}
  For any $\vec \theta \in \R^d$, $\E[\vec \theta \cdot Z(\vec x)] = 0$.
\end{proposition}

\begin{proof}
Fix some $\{i,j\},a$. That component of the dot product will be $(1-\frac{1}{g})\theta_{\{i,j\},a}$ if $x_i - x_j = a$, and will be $-\frac{1}{g} \theta_{\{i,j\},a}$ if $x_i - x_j \neq a$. If $\vec x$ is chosen randomly and uniformly from $(\Z_g)^k$, then $x_i - x_j$ will be a random, uniformly-distributed element of $G$. Thus, the expected value of the $\{i,j\},a$ component of the dot product will be $\frac{1}{g}(1-1/g)\theta_{\{i,j\},a} -\frac{g-1}{g}\frac{1}{g} \theta_{\{i,j\},a} = 0.$ 
\end{proof}

To describe the second moment, we will first define a $d \times d$ matrix $M$ which is indexed in the same way as $\R^d$:
  \begin{equation} \label{matrix_defn}
    M_{(\{i, j\}, a), (\{m, n\}, b)} =
      \begin{cases}
        \frac{g-1}{g^2}, & \textrm{if } \{i, j\} = \{m, n\} \textrm{ and } a = b \\
        -\frac{1}{g^2}, & \textrm{if } \{i,j\} = \{m, n\} \textrm{ and } a \neq b \\
        0, & \textrm{if } \{i,j\} \neq \{m, n\}.
      \end{cases}
  \end{equation}
  
\begin{proposition} \label{second_moment}
  For any $\vec \theta \in \R^d$, $\E[(\vec \theta \cdot Z(\vec x))^2] = \vec \theta^T M \vec \theta$.
\end{proposition}

\begin{proof}
  For convenience of notation, we will let $\vec \xi = Z(\vec x)$; that is, $\vec \xi$ is a vector in $\R^d$ chosen randomly and uniformly from the collection $\{Z(\vec x): \vec x \in (\Z_g)^k\}$. As in the proof of Proposition \ref{first_moment}, for a fixed $\{i,j\},a$, the random variable $\xi_{\{i,j\},a}$ is $1 - 1/g$ with probability $1/g$ and is $-1/g$ with probability $1 - 1/g$. 
  We note that 
  \begin{align*}
     \E[(\vec \theta \cdot \vec \xi)^2] &= \E \left[ \left(\sum_{\{i,j\},a} \theta_{\{i,j\},a} \xi_{\{i,j\},a} \right)^2 \right] \\
     & = \sum_{(\{i,j\},a),(\{m,n\},b)} \theta_{\{i,j\},a} \theta_{\{m,n\},b} \E \left[\xi_{\{i,j\},a}\xi_{\{m,n\},b} \right].
  \end{align*}
Our goal is therefore to show that the expected value $\E[\xi_{\{i,j\},a} \xi_{\{m,n\},b}]$ agrees with the $(\{i,j\},a), (\{m,n\},b)$ entry of $M$.  

We first consider the case where $\{i,j\} = \{m,n\}$ and $a = b$. To compute $\E[(\xi_{\{i,j\},a})^2]$, we recall that the random variable $(\xi_{\{i,j\},a})^2$ is $(1-1/g)^2$ with probability $1/g$ and is $1/g^2$ with probability $1-1/g$; hence, its expectation is $\frac{g-1}{g^2}$, as desired. Next, we consider the case where $\{i,j\} = \{m,n\}$, but $a \neq b$. To compute $\E[\xi_{\{i,j\},a}\xi_{\{i,j\},b}]$, we note that there are two possibilities: if $x_i - x_j \in \{a,b\}$, then $\xi_{\{i,j\},a} \xi_{\{i,j\},b} = -(1-1/g)1/g$, and if $x_i - x_j \not \in \{a,b\}$, then $\xi_{\{i,j\},a} \xi_{\{i,j\},b} = 1/g^2$. The former will occur with probability $2/g$, and the latter will occur with probability $1 - 2/g$. Hence, the expected value is $-\frac{1}{g^2}$.

Finally, if the pairs $\{i,j\}$ and $\{m,n\}$ are not the same, we claim that the variables $\xi_{\{i,j\},a}$ and $\xi_{\{m,n\},b}$ are independent. We will show this by separately considering the cases where $|\{i,j\} \cap \{m,n\}|$ is $0$ or $1$. If $\{i,j\}$ and $\{m,n\}$ are disjoint, then the expressions $x_i - x_j$ and $x_m - x_n$ are clearly independent of one another since $\vec x$ is chosen uniformly from $(\Z_g)^k$. If the pairs are of the form $\{i,j\}$ and $\{i,m\}$, then since $x_i,\ x_j,\ x_m,\ x_i - x_j,$ and $x_i - x_m$ are all uniformly distributed on $\Z_g$ and the first three are independent of each other, then for any $g_1, g_2 \in \Z_g$, we have 
  \begin{align*}
    \PB(x_i - x_j = g_1,\ x_i - x_m = g_2) & = \sum_{c \in Z_g} \PB(x_i = c,\ x_j = c - g_1,\ x_m = c - g_2) \\
    & = \sum_{c \in \Z_g} \PB(x_i = c) \PB(x_j = c-g_1) \PB(x_m = c-g_2) \\
    & = \sum_{c \in \Z_g} \frac{1}{g^3} = \frac{1}{g^2}
  \end{align*}
which is also equal to $\PB(x_i - x_j = g_1) \PB(x_i - x_m = g_2)$. A similar argument can be made for any configuration of $\{i,j\},\{m,n\}$ with exactly one shared element. Hence, whether $|\{i,j\} \cap \{m,n\}|$ is $0$ or $1$, the variables $\xi_{\{i,j\},a}$ and $\xi_{\{m,n\},b}$ are independent, and consequently $\E[\xi_{\{i,j\},a} \xi_{\{m,n\},b}] = \E[\xi_{\{i,j\},a}] \E[\xi_{\{m,n\},b}] = 0$ as in Proposition \ref{first_moment}.
\end{proof}

We will not need to compute the third and fourth moments of $\vec \theta \cdot Z(\vec x)$ explicitly; rather, we will only require estimates of those moments. However, we will need to compute the determinant of $M$. In all that follows, we will use $I_n$ to denote the $n \times n$ identity matrix. We will use the following well-known identity, which can be found (for instance) in \cite{harville}*{Cor.\ 18.1.2}.

\begin{lemma}[Sylvester's Determinant Identity]
  For any $n \times m$ matrix $S$ and $m \times n$ matrix $U$, we have $\det(I_n + SU) = \det(I_m + US).$
\end{lemma}

\begin{proposition} \label{matrix_det}
  With $M$ as defined by \eqref{matrix_defn}, we have $\det(M) = g^{-g \binom{k}{2}}$. 
\end{proposition}

\begin{proof}
We note that $M$ is a block diagonal matrix with $\binom k 2$ repeated copies of the same $(g - 1) \times (g-1)$ submatrix $S$ with $\frac{g-1}{g^2}$ in the diagonal entries and $- \frac{1}{g^2}$ in the off-diagonal entries. We can therefore express $S$ as $g^{-1} I_{g-1} - g^{-2} \vec 1 \vec 1^T$, where $\vec 1$ is the column vector of length $g-1$ consisting of all ones. By Sylvester's Determinant Identity,  $\det(S) = g^{-(g-1)} \det (I_{g-1} - g^{-1} \vec 1 \vec 1^T ) = g^{-(g-1)} \det (1 - g^{-1} \vec 1^T \vec 1 ) = g^{-(g-1)}( 1 - \frac{g-1}{g}) = g^{-g}$; the desired result follows from the block structure of $M$.
\end{proof}

Finally, we state the lemma that gives the desired estimates on $|\Phi(\vec \theta)|$ in the region $B_{\delta}(\vec 0)$.

\begin{lemma} \label{primary_region_estimates}
  Let $\delta > 0$ and $\vec \theta \in B_{\delta}(\vec 0)$. Then there is a function $\epsilon \colon B_{\delta}(\vec 0) \to \R$ such that
  \begin{equation} \label{primary_real_bd}
    \real(\Phi(\vec \theta)) = e^{-\frac{1}{2} \vec \theta^T M \vec \theta} (1 + \epsilon(\vec \theta))
  \end{equation}
and $|\epsilon(\vec \theta)| \leq \frac{1}{6} (d \delta)^4 e^{\frac{1}{2} d^2 \delta^2}$. Moreover,
  \begin{equation} \label{primary_imag_bd}
    |\imag(\Phi(\vec \theta))| \leq \frac{(d \delta)^3}{6}.
  \end{equation}
Further, if $d \delta < 1$, then
  \begin{equation} \label{primary_real_lower}
    \real(\Phi(\vec \theta)) > 1/3.
  \end{equation}
\end{lemma}

\begin{proof}
  We will mimic the proof of Lemma 3.1 of \cite{wdl_levin}. 
  First, using \eqref{taylor_bd_1} with $j = 1$ shows that
    \begin{equation} \label{primary_region_bd_1}
      \left| e^{-\frac{1}{2} \vec \theta^T M \vec \theta} - \left(1 - \frac{1}{2} \vec \theta^T M \vec \theta \right) \right| \leq \frac{1}{8} (\vec \theta^T M \vec \theta)^2.
    \end{equation}
  The coefficients of $M$ are bounded between $-1$ and $1$, and the components of $\vec \theta$ are bounded between $-\delta$ and $\delta$. Hence, by the triangle inequality, 
  \begin{align}
      |\vec \theta^T M \vec \theta| & \leq \sum_{(\{i,j\},a),(\{m,n\},b)} |\theta_{\{i,j\},a} \theta_{\{m,n\},b} M_{(\{i,j\},a),(\{m,n\},b)}| \nonumber \\ 
      & \leq \sum_{(\{i,j\},a),(\{m,n\},b)} \delta^2 = d^2 \delta^2.  \label{primary_region_bd_1a}
  \end{align}
Putting this together with \eqref{primary_region_bd_1} shows that 
  \begin{equation} \label{primary_region_bd_2}
      \left| e^{-\frac{1}{2} \vec \theta^T M \vec \theta} - \left(1 - \frac{1}{2} \vec \theta^T M \vec \theta \right) \right| \leq \frac{1}{8} d^4 \delta^4.
  \end{equation}

Next, let $\vec x \in (\Z_g)^k$. Using \eqref{taylor_bd_2} with $j = 3$ gives
  \begin{equation*}
    \left| e^{i \vec \theta \cdot Z(\vec x)} - \left[ 1 + i \vec \theta \cdot Z(\vec x) - \frac{(\vec \theta \cdot Z(\vec x))^2}{2} - \frac{i (\vec \theta \cdot Z(\vec x))^3}{6} \right]  \right| \leq \frac{1}{24} (\vec \theta \cdot Z(\vec x))^4.
  \end{equation*}
By examining only the real part of the term in the absolute value and recalling that $|\real(z)| \leq |z|$ for any $z \in \C$, we have
  \begin{equation*} 
    \left| \real(e^{i \vec \theta \cdot Z(\vec x)}) - \left[ 1 - \frac{(\vec \theta \cdot Z(\vec x))^2}{2} \right]  \right| \leq \frac{1}{24} (\vec \theta \cdot Z(\vec x))^4.
  \end{equation*}
If $\vec x$ is chosen randomly and uniformly from $(\Z_g)^k$, then this shows that
  \begin{align*}
    & \left| \E \left[ \real(e^{i \vec \theta \cdot Z(\vec x)}) \right] - \E \left[ 1 - \frac{(\vec \theta \cdot Z(\vec x))^2}{2} \right]  \right| \nonumber \\
    & \qquad \leq \E \left| \real(e^{i \vec \theta \cdot Z(\vec x)}) - \left[ 1 - \frac{(\vec \theta \cdot Z(\vec x))^2}{2} \right] \right| \nonumber \\
    & \qquad \leq \frac{1}{24} \E \left[(\vec \theta \cdot Z(\vec x))^4 \right].
  \end{align*}
From the linearity of the $\real$ operator, we have $\E[\real(e^{i \vec \theta \cdot Z(\vec x)})] = \real(\Phi(\vec \theta))$. Thus, Proposition \ref{second_moment} shows that
\begin{equation} \label{primary_region_bd_3}
  \left| \real(\Phi(\vec \theta)) - \left[1 - \frac{1}{2}\vec \theta^T M \vec \theta \right] \right| \leq \frac{1}{24} \E \left[(\vec \theta \cdot Z(\vec x))^4 \right].
\end{equation}

Next, we seek to obtain a similar bound on $\imag(\Phi(\vec \theta))$. Using \eqref{taylor_bd_2} again with $j = 2$ shows that
  \begin{equation*}
    \left| e^{i \vec \theta \cdot Z(\vec x)} - \left[ 1 + i \vec \theta \cdot Z(\vec x) - \frac{1}{2} (\vec \theta \cdot Z(\vec x))^2 \right] \right| \leq \frac{1}{6} | \vec \theta \cdot Z(\vec x)|^3
  \end{equation*}
so examining only the imaginary part in the absolute value and using the fact that $|\imag(z)| \leq |z|$ gives
  \begin{equation*}
    \left| \imag(e^{i \vec \theta \cdot Z(\vec x)}) -  \vec \theta \cdot Z(\vec x) \right| \leq \frac{1}{6} | \vec \theta \cdot Z(\vec x)|^3.
  \end{equation*}
By the same argument as for the real part, if $\vec x \in (\Z_g)^k$ is chosen randomly and uniformly, then 
  \begin{equation*}
    \left| \imag(\Phi(\vec \theta)) - \E[\vec \theta \cdot Z(\vec x)] \right| \leq \frac{1}{6} \E [| \vec \theta \cdot Z(\vec x)|^3]
  \end{equation*}
so by Proposition \ref{first_moment},
  \begin{equation} \label{primary_region_bd_4}
    \left| \imag(\Phi(\vec \theta)) \right| \leq \frac{1}{6} \E [| \vec \theta \cdot Z(\vec x)|^3].
  \end{equation}
  
The next step is to bound the expectations in \eqref{primary_region_bd_3} and \eqref{primary_region_bd_4}. For any $\vec x \in (\Z_g)^k$, the components of $Z(\vec x)$ all have absolute value less than $1$, and the components of $\vec \theta$ all have absolute value at most $\delta$. Hence, 
  \begin{equation} \label{primary_region_bd_5}
    | \vec \theta \cdot Z(\vec x)| \leq \sum_{\{i,j\},a} |\theta_{\{i,j\},a}| \leq d \delta.
  \end{equation}
Combining this with \eqref{primary_region_bd_4} yields \eqref{primary_imag_bd}. Similarly, combining \eqref{primary_region_bd_5} with \eqref{primary_region_bd_3} yields
\begin{equation*} 
  \left| \real(\Phi(\vec \theta)) - \left[1 - \frac{1}{2}\vec \theta^T M \vec \theta \right] \right| \leq \frac{(d \delta)^4}{24}
\end{equation*}
so \eqref{primary_region_bd_2} and the triangle inequality give
\[
  \left| \real(\Phi(\vec \theta)) - e^{-\frac{1}{2} \vec \theta^T M \vec \theta} \right| \leq \frac{(d \delta)^4}{6}.
\]
Dividing both sides by $e^{-\frac{1}{2} \vec \theta^T M \vec \theta}$, then applying \eqref{primary_region_bd_1a} to the right side shows that
\[
  \left| \frac{\real(\Phi(\vec \theta))}{e^{-\frac{1}{2} \vec \theta^T M \vec \theta}} - 1 \right| \leq \frac{1}{6} (d \delta)^4 e^{\frac{1}{2} d^2 \delta^2}.
\]
Therefore, we can write
\[
  \real(\Phi(\vec \theta)) = e^{-\frac{1}{2} \vec \theta^T M \vec \theta} \left[ \frac{\real(\Phi(\vec \theta))}{e^{-\frac{1}{2} \vec \theta^T M \vec \theta}} \right] 
  = e^{-\frac{1}{2} \vec \theta^T M \vec \theta} ( 1 + \epsilon(\vec \theta))
\]
where $|\epsilon(\vec \theta)| \leq \frac{1}{6} (d \delta)^4 e^{\frac{1}{2} d^2 \delta^2}$; this establishes \eqref{primary_real_bd}.

Finally, we observe by \eqref{primary_real_bd} that
  \[
    \real(\Phi(\vec \theta)) \geq e^{-\frac{1}{2} \vec \theta^T M \vec \theta} \left(1 - \frac{1}{6} (d \delta)^4 e^{\frac{1}{2} d^2 \delta^2} \right)
  \]
so by \eqref{primary_region_bd_1a} and the assumption that $d \delta < 1$, we have
  \[
    \real(\Phi(\vec \theta)) \geq \frac{ 1 - \frac{1}{6} (d \delta)^4 e^{\frac{1}{2} d^2 \delta^2}}{e^{\frac{1}{2} \vec \theta^T M \vec \theta}} \geq \frac{1 - \frac{1}{6}e^{1/2}}{e^{1/2}} > 1/3
  \]
which proves \eqref{primary_real_lower}.
\end{proof}

\section{Proof of Main Theorem} \label{S: proof}

Our final task is to put all the pieces together to obtain suitable estimates on the return probability of the random walk. We first gather an assortment of technical lemmas.

\begin{proposition} \label{unsatisfying_transformation}
  There is a symmetric matrix $P$ such that $P^2 = M$. Moreover, there are positive constants $D_1, D_2$ which depend only on $g$ such that for all $\delta > 0$, 
  \[ [-D_1 \delta, D_1 \delta]^d \subset P[-\delta, \delta]^d \subset [-D_2 \delta, D_2 \delta]^d \]
where $P[\delta, \delta]^d = \{P \vec \mu: \vec \mu \in [-\delta, \delta]^d\}.$
\end{proposition}

\begin{proof}
Propositions \ref{second_moment} and \ref{matrix_det} imply that $M$ is positive definite. As noted in the proof of Proposition \ref{matrix_det}, $M$ is a block diagonal matrix with $\binom k 2$ repeated copies of the same $(g-1) \times (g-1)$ submatrix $S = g^{-1} I_{g-1} - g^{-2} \vec 1 \vec 1^T$. It follows that $S$ is also positive definite, so there is a symmetric, positive definite matrix (call it $Q$) such that $Q^2 = S$. The linear transformation corresponding to $Q$ maps $[-1, 1]^{g-1}$ to a nondegenerate subset of $\R^{g-1}$; thus, there are constants $D_1, D_2$ such that $[-D_1, D_1]^{g-1} \subset Q[-1, 1]^{g-1} \subset [-D_2, D_2]^{g-1}$. Since $S$ and $Q$ depend only on $g$, the same is true of $D_1$ and $D_2$.

The block diagonal matrix $P$ consisting of $\binom k 2$ repeated copies of $Q$ is therefore a symmetric, positive definite matrix for which $P^2 = M$. Moreover, it follows that $[-D_1, D_1]^d \subset P[-1, 1]^d \subset [-D_2, D_2]^d$, where $D_1$ and $D_2$ are the same constants as above which depend only on $g$. Since the transformation associated to $P$ is linear, scaling by $\delta$ completes the result.
\end{proof}


For $z \in \C$, we set $\beta(z) = \imag(z)/\real(z)$ and $\alpha(z,t) = 1 - \binom{t}{2} \beta(z)^2$.

\begin{lemma} \label{z^t_to_z}
  Let $t \geq 2$ be an integer, and let $z \in \C$ with $\real(z) > 0$ and $\alpha(z, t) > 0$.
  Then
      \begin{equation} \label{appendix_1}
        \real(z^t) \leq \real(z)^t \left( 1 + \left[ \beta(z) \right]^2 \right)^{t/2}
      \end{equation}      
    and 
      \begin{equation} \label{appendix_2}
        \real(z^t) \geq \real(z)^t \left(1 + \left[ \beta(z) \right]^2 \right)^{t/2} \left(1 + \left[ \frac{t}{\alpha(z,t)} \right]^2 \left[ \beta(z) \right]^2 \right)^{-1/2}.
      \end{equation}
\end{lemma}

\begin{proof}
  This requires only trivial modifications to parts (i) and (iv) of Proposition A.1 in \cite{wdl_levin}.
\end{proof}

\begin{lemma} \label{gaussian_estimate}
  Let $\rho$ be a positive real number.  Then
    \[\sqrt{2 \pi(1 - e^{- \rho^2/2})} \leq \int_{- \rho}^{\rho} e^{-\frac{1}{2} x^2} \ud x \leq \sqrt{2 \pi(1 - e^{-\rho^2})}.\]
\end{lemma}

\begin{proof}
  By multiplying two copies of the integral together, applying Fubini's Theorem, and converting to polar coordinates, we have
  \[\int_0^{\rho} 2 \pi r e^{-\frac{1}{2} r^2} \ud r < \int_{[-\rho, \rho]^2} e^{-\frac{1}{2}(x^2 + y^2)} \ud y \, \ud x < \int_0^{\sqrt{2} \rho} 2 \pi r e^{-\frac{1}{2} r^2} \ud r. \]
Computing the left and right sides and taking square roots gives the result.
\end{proof}

With $D_1, D_2$ as defined in Proposition \ref{unsatisfying_transformation}, we define
  \begin{align*}
    L(g, k, t, \delta) & = [1 + t^2(d \delta)^6]^{-1/2} \left[ 1 - \frac{1}{3}(d \delta)^4 \right]^t [1 - e^{-\frac{t}{2} (D_1\delta)^2}]^{d/2},\\
    U(g, k, t, \delta) & = \left[ 1 + \frac{1}{4}(d \delta)^6 \right]^{t/2} \left[ 1 + \frac{1}{3} ( d \delta)^4 \right]^t  [1 - e^{- t (D_2 \delta)^2}]^{d/2}.
  \end{align*}

\begin{theorem} \label{return_prob_estimates_1}
 Suppose that $\delta < \frac{8}{5} g^{-k-3}k^{-2}$, and let $t$ be any positive integer multiple of $g$ such that $t < 2 (d \delta)^{-3}$. 
If $g$ is odd, or if $g$ and $t/g$ are both even, then
  \begin{equation} \label{return_prob_upper}
    \PB(X_t = \vec 0) \leq \frac{g^{\frac{g}{2} \binom{k}{2} + \binom{k-1}{2}}}{\sqrt{(2 \pi t)^{d}}} U(g , k, t, \delta) + \exp \left({- \frac{11}{192} g^{-k} t \delta^2} \right)
  \end{equation}
and
  \begin{equation} \label{return_prob_lower}
    \PB(X_t = \vec 0) \geq\frac{g^{\frac{g}{2} \binom{k}{2} + \binom{k-1}{2}}}{\sqrt{(2 \pi t)^{d}}} L(g , k, t, \delta) - \exp \left({- \frac{11}{192} g^{-k} t \delta^2}\right)  .
  \end{equation}
\end{theorem}

  In the sequel, $\delta$ will be chosen to vary with $t$ in such a way that the exponential term above will tend to $0$ and the $U$ and $L$ terms will tend to $1$, which will complete the proof of Theorem \ref{main_theorem}.

\begin{proof}
If $g$ is odd or $t/g$ is even, Propositions \ref{useful_integral} and \ref{remainder_bound} show that
\[ \left| \PB(X_t = \vec 0) - \frac{g^{\binom{k-1}{2}}}{(2 \pi)^d} \int_{B_\delta(\vec 0)} \Phi(\vec \theta)^t \ud \vec \theta  \right| \leq \exp \left(- \frac{11}{192} g^{-k} t \delta^2 \right). \]
Therefore, to prove \eqref{return_prob_upper} and \eqref{return_prob_lower}, it will suffice to show that
  \begin{equation} \label{final_theorem_goal_1}
    \frac{L(g,k,t,\delta) g^{\frac{g}{2} \binom{k}{2} }}{\sqrt{(2 \pi t)^{d}}} \leq (2 \pi)^{-d} \int_{B_{\delta}(\vec 0)} \Phi (\vec \theta)^t \ud \vec \theta  \leq \frac{U(g,k,t,\delta) g^{\frac{g}{2} \binom{k}{2}}}{\sqrt{(2 \pi t)^{d} }} .
   \end{equation}
Moreover, since $\Phi(- \vec \theta)$ and $\Phi(\vec \theta)$ are complex conjugates and $B_{\delta}(\vec 0)$ is closed under negation, we have
  \begin{equation} \label{final_theorem_goal_2}
    \int_{B_{\delta}(\vec 0)} \Phi(\vec \theta)^t \ud \vec \theta = \int_{B_{\delta}(\vec 0)} \real(\Phi( \vec \theta)^t) \ud \vec \theta.
  \end{equation}
Our strategy will be to relate $\real(\Phi(\vec \theta)^t)$ to $[\real(\Phi(\vec \theta))]^t$ by using Lemma \ref{z^t_to_z}.

 We note that $\delta < \frac{8}{5} g^{-k-3}k^{-2} < 2 g^{-1} k^{-2} < d^{-1}$; the relationship $d \delta < 1$ will be referenced repeatedly throughout the proof. Using again the definitions $\beta(z) = \imag(z) / \real(z)$ and $\alpha(z,t) = 1 - \binom{t}{2} \beta(z)^2$, we note by \eqref{primary_imag_bd} and \eqref{primary_real_lower} that for $\vec \theta \in B_{\delta}(\vec 0)$ we have
  \begin{equation} \label{beta_upper_bound}
    |\beta(\Phi(\vec \theta))| \leq \frac{(d \delta)^3 / 6}{1/3} = \frac{(d \delta)^3}{2}.
  \end{equation}
Since we assume $t < 2(d \delta)^{-3}$, it follows that $\binom{t}{2} \beta(\Phi(\vec \theta))^2 \leq \binom{t}{2} \frac{(d \delta)^6}{4} < \frac{1}{2}$, whence $\alpha(\Phi(\vec \theta), t) > \frac{1}{2}$.  In particular, since $\alpha(\Phi(\vec \theta), t) > 0$ and since $\real(\Phi(\vec \theta)) > 0$ by \eqref{primary_real_lower}, we can use Lemma \ref{z^t_to_z}.  From \eqref{appendix_1} and \eqref{beta_upper_bound} we have 
  \begin{equation} \label{transition_upper_bound}
    \real(\Phi(\vec \theta)^t) \leq \left[ \real(\Phi(\vec \theta)) \right]^t \left( 1 + \frac{(d \delta)^6}{4} \right)^{t/2}
  \end{equation}
and from \eqref{appendix_2} we have
  \begin{align}
    \real(\Phi(\vec \theta)^t) & \geq \left[\real(\Phi(\vec \theta)) \right]^t \left( 1 + [\beta(\Phi(\vec \theta))] ^2 \right)^{t/2} \left(1 + t^2 \left[\frac{\beta(\Phi(\vec \theta))}{\alpha(\Phi(\vec \theta), t)} \right]^2 \right)^{-\frac{1}{2}}  \nonumber \\
   & \geq \left[\real(\Phi(\vec \theta)) \right]^t \left(1 + t^2 \left[\frac{\beta(\Phi(\vec \theta))}{\alpha(\Phi(\vec \theta), t)} \right]^2 \right)^{-\frac{1}{2}}. \label{transition_lower_bound}
  \end{align}
Since $\alpha(\Phi(\vec \theta), t) \geq 1/2$ and $\beta(\Phi(\vec \theta)) \leq (d \delta)^3 / 2$, it follows that 
  \[ \left[ \frac{\beta(\Phi(\vec \theta))}{\alpha(\Phi(\vec \theta), t)} \right]^2 \leq (d \delta)^6  \]
so \eqref{transition_upper_bound} and \eqref{transition_lower_bound} combine to give
  \begin{align}
    & [1 + t^2 (d \delta)^6 ]^{-1/2} \int_{B_{\delta}(\vec 0)} \left[ \real(\Phi(\vec \theta)) \right]^t \ud \vec \theta \nonumber \\
    & \qquad \quad \leq \int_{B_{\delta}(\vec 0)} \real(\Phi(\vec \theta)^t) \ud \vec \theta \nonumber \\
    & \qquad \quad \leq \left[1 + \frac{(d \delta)^6}{4} \right]^{t/2} \int_{B_{\delta}(\vec 0)} \left[ \real(\Phi(\vec \theta)) \right]^t \ud \vec \theta.  \label{power_traded}
  \end{align} 

We recall from Lemma \ref{primary_region_estimates} that there exists a function $\epsilon: B_{\delta}(\vec 0) \to \R$ such that for $\vec \theta \in B_{\delta}(\vec 0)$,
  \[ \left[\real(\Phi(\vec \theta)) \right]^t = e^{-\frac{t}{2} \vec \theta^T M \vec \theta} ( 1 + \epsilon(\vec \theta))^t \]
and $|\epsilon(\vec \theta)| < \frac{1}{6} (d \delta)^4 e^{\frac{1}{2} (d \delta)^2}$.  Because $d \delta < 1$, it follows that $e^{\frac{1}{2} (d \delta)^2} < 2$, so $|\epsilon(\vec \theta)| < \frac{1}{3}(d \delta)^4$.  Therefore,
  \begin{equation*}
   e^{-\frac{t}{2} \vec \theta^T M \vec \theta} \left[ 1 - \frac{1}{3}(d \delta)^4 \right]^t \leq \left[ \real(\Phi(\vec \theta)) \right]^t  \leq e^{-\frac{t}{2} \vec \theta^T M \vec \theta} \left[ 1 + \frac{1}{3}(d \delta)^4 \right]^t 
  \end{equation*}
and substituting these bounds into \eqref{power_traded} gives
  \begin{align}
   & [1 + t^2 (d \delta)^6 ]^{-1/2}\left[1 - \frac{1}{3}(d \delta)^4 \right]^t \int_{B_{\delta}(\vec 0)}  e^{-\frac{t}{2} \vec \theta^T M \vec \theta} \ud \vec \theta \nonumber \\
   & \qquad \quad \leq \int_{B_{\delta}(\vec 0)} \real(\Phi(\vec \theta)^t) \ud \vec \theta \nonumber \\
   & \qquad \quad \leq \left[1 + \frac{(d \delta)^6}{4} \right]^{t/2} \left[ 1 + \frac{1}{3} (d \delta)^4 \right]^t \int_{B_{\delta}(\vec 0)} e^{-\frac{t}{2} \vec \theta^T M \vec \theta} \ud \vec \theta.  \label{gaussian_integral}
  \end{align} 
To verify \eqref{final_theorem_goal_1} and thus complete the proof, by \eqref{gaussian_integral} and \eqref{final_theorem_goal_2} it suffices to show that 
  \begin{align}
     & [1 - e^{-\frac{t}{2} (D_1\delta)^2}]^{d/2} \left( \frac{2 \pi}{t} \right)^{\frac{d}{2}} g^{\frac{g}{2} \binom{k}{2}} \nonumber \\
     & \qquad \quad \leq \int_{B_{\delta}(\vec 0)} e^{-\frac{t}{2} \vec \theta^T M \vec \theta} \ud \vec \theta \nonumber \\
     & \qquad \quad \leq  [1 - e^{ -t (D_2 \delta)^2}]^{d/2} \left( \frac{2 \pi}{t} \right)^{\frac{d}{2}} g^{\frac{g}{2} \binom{k}{2}} \label{last_reduction_really}
  \end{align}
so we now turn our attention to the integral in the middle.

We recall from Proposition \ref{unsatisfying_transformation} that there is a symmetric matrix $P$ such that $P^2 = M$. Since $B_{\delta}(\vec 0) = [-\delta, \delta]^d$, we have
  \[\int_{B_{\delta}(\vec 0)} e^{-\frac{t}{2} \vec \theta^T M \vec \theta} \ud \vec \theta =  \int_{[-\delta, \delta]^{d}} e^{-\frac{1}{2} (\sqrt{t} P\vec \theta)^T (\sqrt{t} P \vec \theta)} \ud \vec \theta  \]
so if we apply a change of variables with $\vec \eta = \sqrt{t} P \vec \theta$, we have
  \[ \int_{B_{\delta}(\vec 0)} e^{-\frac{t}{2} \vec \theta^T M \vec \theta} \ud \vec \theta =  \frac{1}{t^{d/2}\det(P)} \int_{P[-\sqrt{t} \delta, \sqrt{t} \delta]^{d}} e^{-\frac{1}{2} \vec \eta^T \vec \eta} \ud \vec \eta . \]
Since the integrand is positive, Proposition \ref{unsatisfying_transformation} also implies that
  \begin{align*}
  & \frac{1}{t^{d/2} \det(P)} \int_{[-D_1 \sqrt{t} \delta, D_1 \sqrt{t} \delta]^{d}} e^{-\frac{1}{2} \vec \eta^T \vec \eta} \ud \vec \eta \\
  & \qquad \quad < \int_{B_{\delta}(\vec 0)} e^{-\frac{t}{2} \vec \theta^T M \vec \theta} \ud \vec \theta \\
  & \qquad \quad < \frac{1}{t^{d/2} \det(P)} \int_{[-D_2 \sqrt{t} \delta, D_2 \sqrt{t} \delta]^{d}} e^{-\frac{1}{2} \vec \eta^T \vec \eta} \ud \vec \eta.
  \end{align*}
Because $\vec \eta^T \vec \eta = \sum \eta_{\{i,j\},a}^2$, we can regard the integrals in the lower and upper bounds as the product of $d$ integrals of the form $\int e^{-\frac{1}{2} x^2} \ud x$.  Using the estimates in Lemma \ref{gaussian_estimate} gives
  \begin{align*}
  &  \frac{1}{t^{d/2}\det(P)} \left( \sqrt{2 \pi(1 - e^{- \frac{1}{2} t (D_1 \delta)^2})} \right)^{d} \\
  & \qquad \quad < \int_{B_{\delta}(\vec 0)} e^{-\frac{t}{2} \vec \theta^T M \vec \theta} \ud \vec \theta   \\
  & \qquad \quad < \frac{1}{t^{d/2} \det(P)} \left( \sqrt{2 \pi(1 - e^{- t (D_2 \delta)^2})} \right)^{d} 
  \end{align*}
and since Proposition \ref{matrix_det} shows that $\det(P) = \sqrt{\det(M)} = g^{-\frac{1}{2} g \binom{k}{2}}$, this yields \eqref{last_reduction_really} and completes the proof.
\end{proof}

We pause to remark that if $g$ is even and $t/g$ is odd, then Propositions \ref{useful_integral} and \ref{remainder_bound} show that 
  \( 
    \PB(X_t = \vec 0) \leq \exp \left( - \frac{11}{192} g^{-k} t \delta^2 \right) .
  \)
This hints at the fact that there are no difference matrices over cyclic groups with such parameters \cite{drake}; however, as stated, it does not actually constitute a proof of that result (even asymptotically), since the $g^{kt}$ factor found in \eqref{walk_correspondence} causes the product $g^{kt} \exp (-\frac{11}{192} g^{-k} t \delta^2)$ not to converge to $0$. This result could potentially be obtained by tightening the error term estimate, but such endeavors are not necessary for our purposes.

\begin{proof}[Proof of Theorem \ref{main_theorem}]

Let $g \geq 2$ be fixed. Suppose $(k, t)$ is a sequence of ordered pairs such that $k \geq 3$, $t \to \I$, each value of $t$ is a positive integer multiple of $g$, and there exists some $\varepsilon_0 > 0$ so that $k < (\frac 1 6 - \varepsilon_0) \frac{\log(t)}{\log(g)}$ for all pairs in the sequence. We recall that if $g$ is even and $\lambda = t/g$ is odd, then there is no $(g, k; \lambda)$-difference matrix over $\Z_g$ \cite{drake}; we therefore assume that $g$ is odd or that $\lambda$ is even for every pair in the sequence. We define the sequences $\delta, \varepsilon$ by
$$k =  \left( \frac 1 6 - \varepsilon \right) \frac{\log(t)}{\log(g)}, \qquad \delta = g^{-\frac{5k}{2 - 12 \varepsilon}}.$$
We remark that these definitions imply that $\delta = t^{-5/12}$ and that our assumptions on the sequences $(k, t)$ imply that $\varepsilon \in (\varepsilon_0, 1/6)$. With these definitions, we have three goals: we wish to argue that the hypotheses of Theorem \ref{return_prob_estimates_1} hold for all but finitely many pairs $(k, t)$, that the bracketed components inside $U$ and $L$ tend to $1$, and that  in \eqref{return_prob_upper} and \eqref{return_prob_lower} the exponential error terms become small in comparison to the coefficients on the $U$ and $L$ terms.

First, we verify that $\delta < \frac 8 5 g^{-k-3} k^{-2}$. If the sequence of $k$ values is bounded above, then $\delta = t^{-5/12}$ will certainly be less than $g^{-k-3} k^{-2}$ for sufficiently large $t$. If the sequence of $k$ values is unbounded, then we note that
$$\delta = g^{-\frac{5k}{2 - 12 \varepsilon}} = g^{-\left(k + \frac{3 + 12 \varepsilon}{2 - 12 \varepsilon} k \right)} < g^{-\left(k + \frac 3 2 k + 6 \varepsilon_0 k \right)}$$
and because $k > 2$ we have $\delta < g^{-k-3}g^{-6 \varepsilon_0 k}$, which is smaller than $g^{-k-3} \left(\frac 8 5 k^{-2} \right)$ for sufficiently large $k$. Also, since $k < \log(t)$, we see that
$$t (d \delta)^3 = t (g-1)^3 k^3 (k-1)^3 \delta^3 / 8 < t g^3 [\log(t)]^6 t^{-5/4}/8,$$
which is less than $2$ for sufficiently large $t$. Hence, we have verified that the hypotheses of Theorem \ref{return_prob_estimates_1} hold for all but finitely many pairs $(k, t)$.

Next, we consider the factors of $L$ and $U$. Since $t^2 (d \delta)^6 = (g-1)^6 \binom k 2 ^6 t^{-1/2} < (g-1)^6 (\log(t))^{12} t^{-1/2} \to 0$ as $t \to \I$, we have $[1 + t^2 (d \delta)^2]^{-1/2} \to 1$. On the other hand, because $t (d \delta)^6 < (g-1)^6 (\log(t))^{12} t^{-3/2} \to 0$, it follows that $[1 + \frac 1 4 (d \delta)^6]^{t/2} \to 1$. Similarly, since $t (d \delta)^4 < (g-1)^4 (\log(t))^8 t^{-2/3} \to 0$, the $[1 - \frac 1 3 (d \delta)^4]^t$ and $[1 + \frac 1 3 (d \delta)^4]^t$ terms each tend to $1$. Finally, since the constants $D_1$ and $D_2$ depend only upon $g$, we see that 
$$d e^{-\frac t 2 (D_1 \delta)^2} \leq g (\log(t))^2 e^{- \frac 1 2 D_1^2 t^{1/6}} \to 0, \qquad d e^{-t (D_2 \delta)^2} \leq g (\log(t))^2 e^{- D_2^2 t^{1/6}} \to 0$$
whence $[1 - e^{\frac t 2 (D_1 \delta)^2}]^{d/2}$ and $[1 - e^{-t (D_2 \delta)^2}]^{d/2}$ both tend to $1$. Therefore, as $t \to \I$, $L(g, k, t, t^{-5/12})$ and $U(g, k, t, t^{-5/12})$ both converge to $1$.

Finally, from \eqref{return_prob_upper} and \eqref{return_prob_lower}, we see that 
\begin{align*}
  \liminf_{ t \to \I} \frac{\PB(X_t = \vec 0)}{\left[ \frac{g^{\frac{g}{2} \binom{k}{2} + \binom{k-1}{2}}}{(2 \pi t)^{d/2}} \right]} & \geq \lim_{t \to \I} \left[ L(g, k, t, t^{-\frac{5}{12}}) - \frac{e^{-\frac{11}{192} g^{-k} t \delta^2}}{\left[ \frac{g^{\frac{g}{2} \binom{k}{2} + \binom{k-1}{2}}}{(2 \pi t)^{d/2}} \right]} \right], \\
  \limsup_{ t \to \I} \frac{\PB(X_t = \vec 0)}{\left[ \frac{g^{\frac{g}{2} \binom{k}{2} + \binom{k-1}{2}}}{(2 \pi t)^{d/2}} \right]} & \leq \lim_{t \to \I} \left[ U(g, k, t, t^{-\frac{5}{12}}) + \frac{e^{-\frac{11}{192} g^{-k} t \delta^2}}{\left[ \frac{g^{\frac{g}{2} \binom{k}{2} + \binom{k-1}{2}}}{(2 \pi t)^{d/2}} \right]} \right] 
\end{align*}
and because $L, U \to 1$, if we can show that 
\begin{equation} \label{one_down_and_three_point_six}
  \frac{e^{-\frac{11}{192} g^{-k} t \delta^2}}{\left[ \frac{g^{\frac{g}{2} \binom{k}{2} + \binom{k-1}{2}}}{(2 \pi t)^{d/2}} \right]} \to 0
\end{equation}
then combining the above with \eqref{walk_correspondence} and substituting $t = \lambda g$ will complete the proof. If the sequence of $k$ values is bounded above, then the term in \eqref{one_down_and_three_point_six} is at most
$$ C_1 e^{-C_2 t \delta^2} t^{C_3} = C_1 e^{-C_2 t^{1/6}} t^{C_3}$$
for positive constants $C_1, C_2, C_3$; consequently, this term tends to $0$ as $t \to \I$. On the other hand, suppose that $k \to \I$; in particular, assume that $k \geq 4$. Then
$$\frac{(2 \pi)^{d/2}}{g^{\frac g 2 \binom k 2 + \binom{k-1}{2}}} = \left( \frac{(2 \pi)^{\frac{g-1}{2}}}{g^{\frac g 2 + \frac{k-2}{k}}} \right)^{\binom k 2} \leq \left( \frac{(2 \pi)^{g-1}}{g^{g+1}} \right)^{\binom k 2 / 2}.$$
If $g \geq 7$, then clearly $\frac{(2 \pi)^{g-1}}{g^{g+1}} < 1$, and it can be easily verified that this also holds for $g = 2, \dots, 6$. Hence, the fraction in \eqref{one_down_and_three_point_six} is at most
$t^{d/2} \exp \left(-\frac{11}{192}g^{-k} t \delta^{2}\right),$
and since $g^k < t^{1/6 - \varepsilon_0}$, the aforementioned fraction is at most
$ \exp \left(\frac d 2 \log(t) - \frac{11}{192} t^{\varepsilon_0 - 1/6} t^{1/6}\right) < \exp \left((\log(t))^3/2 - \frac{11}{192} t^{\varepsilon_0} \right)$,
which tends to $0$ as $t \to \I$. This verifies \eqref{one_down_and_three_point_six}, as desired.
\end{proof}

\section{Conclusion} \label{S: conclusion}

By adopting the perspective and tactics of random walks to the problem of difference matrices over cyclic groups, we have developed a formula for the asymptotic number of such matrices as the number of columns (or equivalently, the row inner product $\lambda$) grows large. There are a number of related projects which require further efforts that we leave for future work. One immediate question is what occurs when the underlying cyclic group is replaced with an arbitrary group; most of the changes required to the proof would be to the latter half of Section \ref{S: anatomy}. 

We also note the relationship between difference matrices and orthogonal arrays, as defined in \cite{crc}*{6.1}. Any $(g, k; \lambda)$-difference matrix over $\Z_g$ can be used to construct an $\textrm{OA}_{\lambda}(k,g)$ \cite{crc}*{Rmk 17.7}; conversely, any $\textrm{OA}_{\lambda}(k,g)$ can be viewed as a $(g, k; \lambda g)$-difference matrix over an arbitrary group $G$ of order $g$ \cite{crc}*{Thm 17.10}. These facts intertwine the number of orthogonal arrays and the number of difference matrices in such a way that one can obtain crude estimates on the number of orthogonal arrays. However, these estimates do not yield the exact asymptotics for the number of such arrays, which could be obtained by reinventing the Fourier analysis in this work for those designs. (We remark that Kuperberg, Lovett, and Peled have already completed this analysis for orthogonal arrays with no repeated columns \cite{klp}.)

Another common direction for this type of work is to find bounds that guarantee the existence of $(g, k; \lambda)$-difference matrices over $\Z_g$. For any suitable configuration of $g, k, t, \delta$ such that the expression for $\PB(X_t = \vec 0)$ in \eqref{return_prob_lower} is positive, the existence of a $(g, k; t/g)$-difference matrix over $\Z_g$ is assured. The estimates provided in this work seem not to be sufficient to provide a nontrivial bound of this type, but many of these estimates could be greatly improved with some effort, perhaps to an extent that would yield a nontrivial lower bound on the probability. Existence questions of difference matrices remain an active area of research, and while any bounds on parameters obtained in this way would likely be far from optimal, they may nonetheless be novel. 

In this work, we have also ignored the question of equivalence classes of difference matrices. Commonly, two difference matrices are regarded as equivalent if one can be obtained from another by exchanging rows or columns, rotating an entire row or column by a group element, or applying an automorphism of the underlying group to every element in the matrix. Such actions are difficult to capture with the random walk enumeration scheme described in this paper, and the questions of counting the raw number of matrices and counting the equivalence classes are nontrivially different when rows or columns can be repeated. 

Finally, we remark that enumeration and existence results are each perhaps most interesting in the case when $\lambda$ is small. To illustrate, we recall that in order for a $(g, 3; 1)$-difference matrix over $\Z_g$ to exist, it is necessary for $g$ to be odd \cite{drake}; however, this condition is also sufficient \cite{ge}. The general question of the existence of a $(g, k; 1)$-difference matrix over $\Z_g$ with $g$ odd and $k \geq 4$ remains open. For example, it is known that both a $(5, 4; 1)$-difference matrix and a $(7, 4; 1)$-difference matrix over cyclic groups exist \cite{evans}, and a computer search has shown that a $(9, 4; 1)$-difference matrix over $\Z_9$ does not \cite{ge}, which illustrates that the $k = 4$ case is not as tidy as the $k = 3$ case. In principle, combining \eqref{fourier_inversion} with \eqref{walk_correspondence} shows that one can obtain the the exact number of $(g, k; \lambda)$-difference matrices over $\Z_g$ by evaluating an integral, so one might hope to obtain interesting results about relatively small $\lambda$ by minimizing all the error terms in the preceding Fourier analysis. Analogously, one might hope to use this sort of tactic to resolve the Hadamard conjecture, or to count the number of Steiner triple system incidence matrices. Of course, such efforts have thus far fallen short of those lofty goals, but interesting existence and enumeration results of this nature have been derived for other combinatorial designs. While it is likely too much to hope that this Fourier analysis can address the enumeration or existence of $(g, k; \lambda)$-difference matrices over $\Z_g$ when $k = g \lambda$, perhaps it can yield results when $\lambda$ grows slowly as a function of $k$.

\section{Acknowledgements}

The author would like to extend sincere appreciation to the anonymous reviewer of this work for the immensely helpful comments and careful attention to detail.

\bibliography{montgomery}

\end{document}